\documentclass{amsart}
\usepackage[utf8]{inputenc}

\usepackage{geometry}
\usepackage{amssymb,latexsym,enumitem,bbm,xcolor}
\usepackage[noadjust]{cite}
\usepackage{setspace}
\usepackage{color}
\usepackage{hyperref}
\usepackage{bbm}

\newcommand{\R}{\mathbb{R}}
\newcommand{\pp}{\mathbb{P}}
\newcommand{\C}{\mathbb{C}}

\newcommand{\Z}{\mathbb{Z}}
\newcommand{\F}{\mathcal{F}}

\newcommand{\Sc}{\mathcal{S}}

\newcommand{\Exp}{\mathbb{E}}

\newcommand{\inpr}[3][]{\left\langle#2 \,,\, #3\right\rangle_{#1}}

\numberwithin{equation}{section}

\newtheorem{theorem}{Theorem}[section]
\newtheorem{corollary}[theorem]{Corollary}
\newtheorem{lemma}[theorem]{Lemma}

\newtheorem{definition}[theorem]{Definition}

\allowdisplaybreaks

\title[Stochastic PDE with bilaplacian]{Stochastic PDEs involving a bilaplacian operator}

\author{Suprio Bhar}
\address{Suprio Bhar, Department of Mathematics and Statistics, Indian Institute of Technology Kanpur, Kalyanpur, Kanpur - 208016, India.}
\email{suprio@iitk.ac.in}

\author{Barun Sarkar}
\address{Barun Sarkar, Department of Mathematics, Indian Institute of Technology Madras, Chennai - 600036, India.}
\email{barun@iitm.ac.in}

\begin{document}

\begin{abstract}
In this article, we study the existence and uniqueness problem for linear Stochastic PDEs involving a bilaplacian operator. Our results on the existence and uniqueness are obtained through an application of a Monotonicity inequality, which we also prove here. As an application of these results, we also obtain a probabilistic representation of the solution for a linear PDE involving the bilaplacian operator.
\end{abstract}

\maketitle

\section{Introduction}\label{S1:intro}

Let $\Sc$ be the space of rapidly decreasing smooth functions on $\R$, called the Schwartz space (due to L. Schwartz \cite{MR0035918}) and $\Sc^\prime$ denote the dual space, called the space of tempered distributions. Let $(\Omega, \F, (\F_t)_{t \geq 0}, \pp)$ be a filtered probability space satisfying the usual conditions. We consider the following Stochastic PDE (SPDE) in $\Sc^\prime$: 

\begin{equation}\label{bilapspdeß1}
dX_t =  L(X_t)\, dt +   A(X_t)\cdot dB_t, t\geq 0;\ \ X_0= \Psi,
\end{equation}
and the associated PDE in $\Sc^\prime$:
\begin{equation}\label{bilappde}
\frac{\partial}{\partial t} u_t =  L(u_t), t \geq 0;\ \ u_0= \Psi,
\end{equation}
where, 
\begin{enumerate}[label=(\roman*)]
    \item $\Psi \in \Sc^\prime$ is $\F_0$-measurable random variable,
    \item $\{B_t\}_{t\geq0}$ is a $2$-dimensional standard Brownian motion with components given by $B_t :=  \begin{pmatrix}
        B_t^1\\B_t^2
    \end{pmatrix}$,
    
\item $L: \Sc^\prime \to \Sc^\prime, A: \Sc^\prime \to \Sc^\prime \times \Sc^\prime$ are linear differential operators defined as follows: 
\begin{equation}\label{L-operator}
L(\phi) := -\frac{\kappa^2}{2}\, \partial^4\, \phi + \frac{\sigma^2}{2}\, \partial^2\, \phi -b\, \partial\, \phi,
\end{equation}
and $A(\phi) := (A_1 (\phi), A_2 (\phi))$, with $A_1,A_2:\Sc'\to\Sc'$, such that
\begin{equation}\label{A-operator}
A_1(\phi) := -\sigma \partial \phi,\quad A_2(\phi) := \kappa \partial^2 \phi,
\end{equation}

\item $\kappa, \sigma, b$ are real constants,

\item $AX_t\cdot\, dB_t = A_1 X_t\,  dB^1_t + A_2 X_t\, dB^2_t$.

\end{enumerate}

In this paper, we prove the existence and uniqueness of strong solution of the Stochastic PDE \eqref{bilapspdeß1} (see Theorem \ref{exist-unique-soln-spde}) and the PDE \eqref{bilappde} (see Theorem \ref{exist-unique-soln-pde}) in $\Sc^\prime$. Moreover, we also obtain a stochastic representation of the solution to the PDE \ref{exist-unique-soln-pde}. Our proof heavily relies on the Monotonicity inequality for $(L, A)$ (see Corollary \ref{L-A-Monotonicity}), involving a bilaplacian operator. Now, we present a brief literature survey related to our discussion.

\begin{enumerate}[label=(\Roman*)]
    \item Bilaplacian operator appears in many practical models, for example, membrane models \cite{MR2510021,MR4038046,MR4115023,MR4215327}, sandpiles model \cite{MR3877548}  and the references therein.

    \item The Monotonicity inequality, originally introduced in \cite{MR570795}, has been of considerable interest in the study of Stochastic PDEs \cite{MR3063763,MR2479730,MR4117986,MR4148176,MR2590157, MR1465436, MR2373102}. In the case when $L$ is a second order linear differential operator, a Monotonicity inequality for $(L, A)$ was first established for constant coefficient differential operators in \cite{MR2590157} and then for some variable coefficients in \cite{MR3331916, MR4117986}. A related inequality, called the Coercivity inequality, has also been part of the Stochastic PDE literature when using the variational methods (see \cite{MR2560625, MR3410409, MR0651582, MR553909, MR570795, MR2001996, MR2460554, MR2512800}).

    \item Probabilistic representations of solutions to the heat equation in $\Sc'$ has been studied in \cite{MR1999259}, here we study the probabilistic representations of solutions to a fourth order PDE. There are other well-known probabilistic representations of solutions to PDEs, for example see the Feynman-Kac type representation to the Cauchy Problem \cite[Theorem 7.6]{MR1121940} and the Kolmogorov Backward equation \cite[Chapter 3.6]{MR2560625}.

    \item Stochastic PDEs, and in general Stochastic Analysis have applications to many areas of study such as Statistical Physics, Financial Mathematics and Mathematical Biology (see \cite{MR2560625,MR3236753, MR3410409, MR1465436, MR570795, MR771478, MR1472487}).
\end{enumerate}

Our results have been stated (and proved) in one-dimensional setting, for example the tempered distributions we consider are on $\R$. It may be possible to extend these results to higher dimensions involving tempered distributions on $\R^d$.

This article is organized as follows. In Section \ref{S2:main-results}, we describe the notations and main results. Section \ref{S3:proof-monotonicity} is devoted to the proof of Monotonicity inequality for $(L, A)$ (see Theorem \ref{4th-order-monotonicity} and Corollary \ref{L-A-Monotonicity}) and in Section \ref{S4:proof-existence-uniqueness}, we apply the Monotonicity inequality to obtain the existence and uniqueness of strong solution of the Stochastic PDE \eqref{bilapspdeß1} (see Theorem \ref{exist-unique-soln-spde}) and the PDE \eqref{bilappde} (see Theorem \ref{exist-unique-soln-pde}) along with the probabilistic representation of the said PDE.

\section{Notations and Main results}\label{S2:main-results}
\subsection{Topology on Schwartz space}\label{S2:topology}
For $p \in \R$, consider the increasing norms $\|\cdot\|_p$, defined by the inner
products (see \cite{MR771478})
\[\inpr[p]{f}{g}:=\sum_{k=0}^{\infty}(2k + 1)^{2p} \inpr[0]{f}{h_k} \inpr[0]{g}{h_k} ,\ \ \ f,g\in\Sc.\]
Here, $\inpr[0]{\cdot}{\cdot}$ is the usual
inner product in $\mathcal{L}^2(\R)$ and $\{h_k\}_{k = 0}^{\infty}$ is an orthonormal basis for $\mathcal{L}^2(\R)$ given by the Hermite functions \[h_k(t)=(2^k k! \sqrt{\pi})^{-1/2}\exp\{-t^2/2\}H_k(t),\] where $H_k$'s are the Hermite polynomials. We define the Hermite-Sobolev spaces $\Sc_p, p \in \R$ as the completion of $\Sc$ in
$\|\cdot\|_p$. Note that the dual space $\Sc_p^\prime$ is isometrically isomorphic with $\Sc_{-p}$ for $p\geq 0$. We also have $\Sc = \bigcap_{p}(\Sc_p,\|\cdot\|_p), \Sc^\prime=\bigcup_{p>0}(\Sc_{-p},\|\cdot\|_{-p})$ and $\Sc_0 = \mathcal{L}^2(\R)$. We state the natural inclusions involving these spaces. For $0<q<p$, \[\Sc \subset \Sc_p \subset \Sc_q \subset \mathcal L^2(\R) = \Sc_0 \subset \Sc_{-q}\subset \Sc_{-p}\subset \Sc^\prime.\]

We also recall the distributional derivative operator on the space of tempered distributions. Consider the derivative map denoted by $\partial = \frac{d}{dx}:\Sc \to
\Sc$. We can extend this map by duality to
$\partial:\Sc^\prime \to \Sc^\prime$ as follows: for $\psi \in
\Sc^\prime$, $\partial \psi \in \Sc^\prime$ is defined by
\[\inpr{\partial \psi}{\phi} := -\inpr{\psi}{\partial \phi}, \; \forall \phi
\in \Sc.\]
The operator $\partial$ acts on the basis of Hermite functions as follows (see \cite[Appendix A.5, equation (A.26)]{MR562914}):
\begin{equation}\label{derivative-recursion}
\partial\, h_n(x) = \sqrt{\frac{n}{2}}\, h_{n-1}(x) - \sqrt{\frac{n+1}{2}}\, h_{n+1}(x), \, \forall n \in \Z_+.
\end{equation}
For notational convenience, we adopt the notation that $h_n \equiv 0$ whenever $n < 0$. As a consequence of the above recurrence relation, we have that $\partial : \Sc_p \to \Sc_{p - \frac{1}{2}}$ is a bounded linear operator for any $p \in \R$. Moreover, $\partial^2 : \Sc_{p} \to \Sc_{p - 1}$ and $\partial^4 : \Sc_{p} \to \Sc_{p - 2}$ are also bounded linear operators for any $p \in \R$. In particular, we have the following boundedness of the operators $L$ and $A$ (as in \eqref{L-operator} and \eqref{A-operator}).

\begin{lemma}\label{L-A-bound-operator}
The operators $A_1: \Sc_p \to \Sc_{p - \frac{1}{2}}, A_2: \Sc_p \to \Sc_{p - 1}$ and $L:\Sc_p \to \Sc_{p - 2}$ is bounded for any $p \in \R$. In particular, all these operators are bounded if considered as mappings from $\Sc_p$ to $\Sc_{p - 2}$ for any $p \in \R$.
\end{lemma}

\subsection{Monotonicity inequality}\label{S2:Monotonicity-inequality}
In this subsection, we discuss new results concerning the Monotonicity inequality involving the operators $(L, A)$. The next result is the first main result of this article.

\begin{theorem}[Main result 1]\label{4th-order-monotonicity}
Fix $p \in \R$. Then, there exists a constant $C = C(p) > 0$, such that 
\begin{equation}\label{monotonicity-inequality}
- \inpr[p]{\phi}{\partial^4\phi} + \|\partial^2\phi\|_p^2 \leq C \|\phi\|_p^2, \,\forall \phi \in \Sc.
\end{equation}
Moreover, by density arguments, the inequality is true for all $\phi \in \Sc_{p + 2}$.
\end{theorem}

Note that the case $p = 0$ in the above theorem follows from an integration by parts argument. 

As a consequence of the above result and \cite[Theorem 2.1]{MR2590157}, we obtain the following inequality involving $L$ and $A$.

\begin{corollary}[Monotonicity inequality for $(L, A)$]\label{L-A-Monotonicity}
Fix $p \in \R$. Then, there exists a constant $C = C(p, \kappa, \sigma, b) > 0$, such that 
\begin{equation}\label{monotonicity-inequality-L-A}
 2\inpr[p]{\phi}{L\phi} + \|A\phi\|_{HS(p)}^2 \leq C \|\phi\|_p^2, \,\forall \phi \in \Sc,
\end{equation}
where $\|A\phi\|_{HS(p)}^2 := \|A_1\phi\|_p^2 + \|A_2\phi\|_p^2, \forall \phi \in \Sc$.
Moreover, by density arguments, the inequality is true for all $\phi \in \Sc_{p + 2}$.
\end{corollary}

Proofs of these results are discussed in Section \ref{S3:proof-monotonicity}.

\subsection{Applications to Stochastic PDEs}\label{S2:stochastic-pde}
We state the notion of solution used in this article.

\begin{definition}
Let $p \in \R$ and $\Psi \in \Sc_p$. We say that an $(\F_t)_{t \geq 0}$ adapted $\Sc_p$ valued continuous process $\{X_t\}_t$ is a strong solution of the Stochastic PDE \ref{bilapspdeß1} if it satisfies the equality 
\begin{equation}\label{integral-equality-spde}
X_t =  \Psi + \int_0^t L(X_s)\, ds +   \int_0^t A(X_s)\cdot dB_s, t \geq 0  
\end{equation}
in some $\Sc_q$ with $q \leq p$. In this case, we say that $\{X_t\}_t$ is an $\Sc_p$ valued strong solution of \eqref{bilapspdeß1} with equality in $\Sc_q$.
\end{definition}

We have the following existence and uniqueness for the Stochastic PDE \eqref{bilapspdeß1}. The proof is discussed in Section \ref{S4:proof-existence-uniqueness}.

\begin{theorem}[Main result 2]\label{exist-unique-soln-spde}
Let $p \in \R$ and $\Psi \in \Sc_p$ such that $\Exp\|\Psi\|_p^2<\infty$. Then, there exists a unique $\Sc_p$ valued solution of the Stochastic PDE \eqref{bilapspdeß1} with equality in $\Sc_{p - 2}$.
\end{theorem}

\subsection{Probabilistic representation for the fourth-order PDE}\label{S2:pde}
We introduce the notion of solution for the PDE used in this article.

\begin{definition}
Let $p \in \R$ and $\Psi \in \Sc_p$. We say that an $\Sc_p$ valued continuous $\{u_t\}_t$ is a strong solution of the PDE \ref{bilappde} if it satisfies the equality 
\begin{equation}\label{integral-equality-pde}
u_t =  \Psi + \int_0^t L(u_s)\, ds, t \geq 0
\end{equation}
in some $\Sc_q$ with $q \leq p$. In this case, we say that $\{u_t\}_t$ is an $\Sc_p$ valued strong solution of \eqref{bilappde} with equality in $\Sc_q$.
\end{definition}

We have the following existence and uniqueness for the PDE \eqref{bilappde}. To the best of our knowledge, the probabilistic representation for the solution is a new result. The proof is discussed in Section \ref{S4:proof-existence-uniqueness}.

\begin{theorem}[Main result 3]\label{exist-unique-soln-pde}
Let $p \in \R$ and the initial condition $\Psi \in \Sc_p$ is deterministic. Then, there exists a unique $\Sc_p$ valued strong solution $\{u_t\}_t$ of the PDE \eqref{bilappde} with equality in $\Sc_{p - 2}$. Moreover, $u_t = \Exp X_t, \forall t \geq 0$, where $\{X_t\}_t$ is as in Theorem \ref{exist-unique-soln-spde}.
\end{theorem}

\section{Proof of Monotonicity inequality (Theorem \ref{4th-order-monotonicity} and Corollary \ref{L-A-Monotonicity})}\label{S3:proof-monotonicity}
This section is devoted to the proof of Theorem \ref{4th-order-monotonicity}. We require several lemmas for our final argument.

Consider the linear operators $U_m, m \in \Z_+$ on $\Sc$ described on the basis of Hermite functions as follows:
 \[U_m h_n := h_{n-m}, \forall n \geq 0, m \geq 0.\]
 We continue to use the notational convention that $h_n \equiv 0$ whenever $n < 0$. Arguing as in \cite[Proof of Theorem 2.1]{MR3331916}, we conclude the following result.

\begin{lemma}\label{shift-operator}
$U_m, m \in \Z_+$ are bounded linear operators on $(\Sc, \|\cdot\|_p)$ and hence extends to bounded linear operators, again denoted by $U_m$, on $\Sc_p$.
\end{lemma}

Iterating the relation \eqref{derivative-recursion}, we obtain the following result. 
\begin{lemma}\label{derivative-expansion}
We have,

\begin{equation}
\begin{split}
\partial^2\, h_n(x) &= \frac{\sqrt{n(n-1)}}{2}\, h_{n-2}(x) - \frac{2n+1}{2}\, h_n(x) + \frac{\sqrt{(n+1)(n+2)}}{2}\, \, h_{n+2}(x),\\
\partial^3\, h_n(x) &= \frac{\sqrt{n(n-1)(n-2)}}{2\sqrt{2}}\, h_{n-3}(x)  - \frac{3n\sqrt{n}}{2\sqrt{2}}\, h_{n-1}(x) \\
& \quad + \frac{3(n+1)\sqrt{n+1}}{2\sqrt{2}}\, h_{n+1}(x) - \frac{\sqrt{(n+1)(n+2)(n+3)}}{2\sqrt{2}}\, \, h_{n+3}(x),\\
\partial^4\, h_n(x) &= \frac{\sqrt{n(n-1)(n-2)(n-3)}}{4}\, h_{n-4}(x) - \frac{(2n-1)\sqrt{n(n-1)}}{2}\, h_{n-2}(x) \\
& \quad + \frac{3n^2+ 3(n+1)^2}{4}\, h_n(x) - \frac{(2n+3)\sqrt{(n+1)(n+2)}}{2}\, h_{n+2}(x) \\
& \quad + \frac{\sqrt{(n+1)(n+2)(n+3)(n+4)}}{4}\, \, h_{n+4}(x).
\end{split}
\end{equation}

\end{lemma}

We shall also require the following result. We choose an analytic branch of $z \mapsto z^{2p}$ in a domain containing the positive real axis.

\begin{lemma}\label{fn-bound-lemma}
We consider the following functions $f_j, j = 1, 2, 3, 4$ defined on a neighbourhood of $0$, say $B(0, \delta) = \{z \in \C: |z| < \delta\}$ for some $\delta > 0$, sufficiently small. We take
\begin{align*}
f_1(z) &:= \left( \frac{2 - 3z}{2 + z} \right)^{2p} + \left( \frac{2 + 5z}{2 + z} \right)^{2p} - 2,\\
f_2(z) &:=  2\left( \frac{2 + 5z}{2 + z} \right)^{2p} - 1 - \left( \frac{2 + 9z}{2 + z} \right)^{2p},\\
f_3(z) &:= -\left( \frac{2 - 3z}{2 + z} \right)^{2p} + 3\left( \frac{2 + 5z}{2 + z} \right)^{2p} - 2,\\
f_4(z) &:= \left( \frac{2 + 5z}{2 + z} \right)^{2p} - 1.
\end{align*}
Then, there exists analytic functions $g_j, j = 1, 2, 3, 4$ on $B(0, \delta)$ with $g_j(0) \neq 0, j = 1, 2, 3, 4$ such that for all $z \in B(0, \delta)$, $f_j(z) = z^2 g_j(z), j = 1, 2$ and $f_j(z) = zg_j(z), j = 3, 4$. In particular, there exists a constant $C > 0$ such that for large positive integers $k$, we have
\begin{equation}\label{fn-bounds}
\left|f_1\left(\frac{1}{k}\right)\right| \leq \frac{C}{k^2},\quad \left|f_2\left(\frac{1}{k}\right)\right| \leq \frac{C}{k^2},\quad \left|f_3\left(\frac{1}{k}\right)\right| \leq \frac{C}{k}, \quad \left|f_4\left(\frac{1}{k}\right)\right| \leq \frac{C}{k}.  
\end{equation}
\end{lemma}

\begin{proof}
The argument is similar to \cite[Lemma 2.2]{MR2590157}.

We note that $f_1(0) = f_1^\prime(0) = f_2(0) = f_2^\prime(0) = f_3(0) = f_4(0) = 0$ with $f_1^{\prime\prime}(0) \neq 0, f_2^{\prime\prime}(0) \neq 0, f_3^\prime(0) \neq 0$ and $f_4^\prime(0) \neq 0$. Consequently, we get the existence of $g_j$'s such that for all $z \in B(0, \delta)$, $f_j(z) = z^2 g_j(z), j = 1, 2$ and $f_j(z) = zg_j(z), j = 3, 4$ with $g_j(0) \neq 0, j = 1, 2, 3, 4$. Using the bounds for $g_j$'s on $\overline{B(0, \frac{\delta}{2})} = \{z \in \C: |z| \leq \frac{\delta}{2}\}$, we get \eqref{fn-bounds}. 
\end{proof}

\begin{proof}[Proof of Theorem \ref{4th-order-monotonicity}]
We organize the proof by splitting it in three steps. In the first two steps we expand the terms $\inpr[p]{\phi}{\partial^4\phi}$ and $\|\partial^2\phi\|_p^2$, respectively, on the basis of Hermite functions. In the third and final step, we combine all the intermediate computations and draw the conclusion.

\textbf{Step 1:} We start with the term $\inpr[p]{\phi}{\partial^4\phi}$. Any $\phi\in\Sc$ can be expressed as  $\phi=\sum_{n=0}^\infty \phi_nh_n$, with $\phi_n\in\R, \forall n$. We use the notational convention that $\phi_n \equiv 0$ whenever $n < 0$.  By Lemma \ref{derivative-expansion}, we have
 \begin{align*}
\partial^4\phi &= \sum_{n=0}^\infty \phi_n \left(\partial^4 h_n\right)\\
 & = \sum_{n=0}^\infty \phi_n \Big\{ \frac{\sqrt{n(n-1)(n-2)(n-3)}}{4}\, h_{n-4} - \frac{(2n-1)\sqrt{n(n-1)}}{2}\, h_{n-2} \\
& \hspace{1.8cm} + \frac{3n^2+ 3(n+1)^2}{4}\, h_n - \frac{(2n+3)\sqrt{(n+1)(n+2)}}{2}\, h_{n+2} \\
& \hspace{1.8cm} + \frac{\sqrt{(n+1)(n+2)(n+3)(n+4)}}{4}\, \, h_{n+4} \Big\} \\
& = \sum_{n=0}^\infty \phi_n \left\{ U_4 A_4 h_{n} + U_2 A_2 h_n + U_0 A_0 h_n + U_{-2} A_{-2} h_n + U_{-4} A_{-4} h_n\right\}, \\
 \end{align*}
where, the linear operators $A_4, A_2, A_0, A_{-2}$ and $A_{-4}$ on $\Sc$ are described on the basis of Hermite functions as follows:
\begin{align*}
& A_4 h_n := \frac{\sqrt{n(n-1)(n-2)(n-3)}}{4} h_n,\\
& A_2 h_n := - \frac{(2n-1)\sqrt{n(n-1)}}{2} h_n,\\
& A_0 h_n := \frac{3n^2+ 3(n+1)^2}{4}h_n,\\
& A_{-2} h_n := - \frac{(2n+3)\sqrt{(n+1)(n+2)}}{2} h_n,\\
& A_{-4} h_n := \frac{\sqrt{(n+1)(n+2)(n+3)(n+4)}}{4} h_n.
\end{align*}
Therefore,
\begin{align}\label{del4Mnph}
 \begin{split}
 & \inpr[p]{\phi}{\partial^4\phi} \\
 & = \sum_{k,m=0}^\infty \left\{\inpr[p]{\phi_kh_k}{\phi_m U_4 A_4 h_{m}} + \inpr[p]{\phi_kh_k}{\phi_m U_2 A_2 h_{m}}\right\} \\
 & \quad + \sum_{k,m=0}^\infty \inpr[p]{\phi_kh_k}{\phi_m U_0 A_0 h_{m}} \\
 & \quad + \sum_{k,m=0}^\infty \left\{ \inpr[p]{\phi_kh_k}{\phi_m U_{-2} A_{-2} h_{m}} + \inpr[p]{\phi_kh_k}{\phi_m U_{-4} A_{-4} h_{m}} \right\}
\end{split}
\end{align}

Now, consider the terms of \eqref{del4Mnph} individually.
\begin{enumerate}[label=(\Roman*)]
 \item 
 \begin{align}\label{estmt101}
 \begin{split}
 & \sum_{k,m=0}^\infty \inpr[p]{\phi_kh_k}{\phi_m U_4 A_{4} h_{m}} \\
 & =\sum_{k,m,i=0}^\infty (2i+1)^{2p} \inpr[]{\phi_kh_k}{h_i} \inpr[]{\phi_m U_4 A_4 h_{m}}{h_i}\\
 & = \sum_{i=0}^\infty  (2i+1)^{2p}\, \phi_i\phi_{i+4}\, \frac{\sqrt{(i+4)(i+3)(i+2)(i+1)}}{4}
 \end{split}
 \end{align}

 \item 
\begin{align}\label{estmt102}
 \begin{split}
 & \sum_{k,m=0}^\infty \inpr[p]{\phi_kh_k}{\phi_m U_2 A_2 h_{m}}\\
 & =\sum_{k,m,i=0}^\infty (2i+1)^{2p} \inpr[]{\phi_kh_k}{h_i} \inpr[]{\phi_m U_2 A_2 h_{m}}{h_i} \\
 & = - \sum_{i=0}^\infty (2i+1)^{2p} \phi_i \phi_{i+2} \frac{(2i+3)\sqrt{(i+2)(i+1)}}{2}
 \end{split}
 \end{align}

 \item 
\begin{align}\label{estmt103}
 \begin{split}
& \sum_{k,m=0}^\infty \inpr[p]{\phi_kh_k}{\phi_m U_0 A_0 h_{m}} \\
& =\sum_{k,m,i=0}^\infty (2i+1)^{2p} \inpr[]{\phi_kh_k}{h_i} \inpr[]{\phi_m U_0 A_0 h_{m}}{h_i} \\
 & = \sum_{i=0}^\infty (2i+1)^{2p} \phi_i^2 \frac{3i^2+ 3(i+1)^2}{4}
 \end{split}
 \end{align}

 \item 
\begin{align}\label{estmt104}
  \begin{split}
  & \sum_{k,m=0}^\infty \inpr[p]{\phi_kh_k}{\phi_m U_{-2} A_{-2} h_{m}}\\
 & =\sum_{k,m,i=0}^\infty (2i+1)^{2p} \inpr[]{\phi_kh_k}{h_i} \inpr[]{\phi_m U_{-2} A_{-2} h_{m}}{h_i} \\
& = - \sum_{i=0}^\infty (2i+1)^{2p} \phi_i \phi_{i-2} \frac{(2i-1)\sqrt{i(i-1)}}{2} \\
& = - \sum_{l=0}^\infty (2l+5)^{2p} \phi_{l+2} \phi_l \frac{(2l+3)\sqrt{(l+2)(l+1)}}{2}\ \ \text{[putting $i-2=l$]} 
 \end{split}
 \end{align}

  \item 
\begin{align}\label{estmt105}
 \begin{split}
 & \sum_{k,m=0}^\infty \inpr[p]{\phi_kh_k}{\phi_m U_{-4} A_{-4} h_{m}} \\
 & =\sum_{k,m,i=0}^\infty (2i+1)^{2p} \inpr[]{\phi_kh_k}{h_i} \inpr[]{\phi_m U_{-4} A_{-4} h_{m}}{h_i} \\
& = \sum_{i=0}^\infty (2i+1)^{2p} \phi_i \phi_{i-4} \frac{\sqrt{(i-3)(i-2)(i-1)i}}{4} \\
& = \sum_{l=0}^\infty (2l+9)^{2p} \phi_{l+4} \phi_l \frac{\sqrt{(l+1)(l+2)(l+3)(l+4)}}{4}\ \ \text{[putting $i-4=l$]}
 \end{split}
 \end{align} 

 \end{enumerate}

\textbf{Step 2:} We now look at the term $\|\partial^2\phi\|_p^2$. By Lemma \ref{derivative-expansion},\begin{align*}
 \partial^2\phi & = \sum_{n=0}^\infty \phi_n \left(\partial^2 h_n\right) \\
 & = \sum_{n=0}^\infty \phi_n \left\{ \frac{\sqrt{n(n-1)}}{2}\, h_{n-2} - \frac{2n+1}{2}\, h_n + \frac{\sqrt{(n+1)(n+2)}}{2}\, \, h_{n+2} \right\} \\
 & = \sum_{n=0}^\infty \phi_n \left\{U_2 B_2 h_n + U_0 B_0 h_n + U_{-2} B_{-2} h_n\right\},
\end{align*}
where, the linear operators $B_2, B_0$ and $B_{-2}$ on $\Sc$ are described on the basis of Hermite functions as follows:
\begin{align*}
 & B_2 h_n := \frac{\sqrt{n(n-1)}}{2} h_n,\\
 & B_0 h_n := - \frac{2n+1}{2} h_n,\\
 & B_{-2} h_n:= \frac{\sqrt{(n+1)(n+2)}}{2} h_n.
\end{align*}
Therefore,
\begin{align}\label{del22Mnph}
 \begin{split}
\|\partial^2\phi\|_p^2  &= \inpr[p]{\partial^2\phi}{\partial^2\phi} \\
 & = \sum_{k,m=0}^\infty \Big\{ \inpr[p]{\phi_k U_2 B_2 h_k}{\phi_m U_2 B_2 h_m} + \inpr[p]{\phi_k U_0 B_0 h_k}{\phi_m U_0 B_0 h_m}  \\
 & \hspace{2cm} + \inpr[p]{\phi_kU_{-2} B_{-2} h_k}{\phi_m U_{-2}B_{-2} h_m} \Big\}\\
 & \quad + \sum_{k,m=0}^\infty \left\{ \inpr[p]{\phi_k U_2 B_2 h_k}{\phi_m U_0 B_0 h_m} + \inpr[p]{\phi_k U_2 B_2 h_k}{\phi_m U_{-2} B_{-2} h_m} \right\} \\
 & \quad + \sum_{k,m=0}^\infty \left\{ \inpr[p]{\phi_k U_0 B_0 h_k}{\phi_m U_2 B_2 h_m} + \inpr[p]{\phi_k U_0 B_0 h_k}{\phi_m U_{-2} B_{-2} h_m} \right\} \\
 & \quad + \sum_{k,m=0}^\infty \left\{ \inpr[p]{\phi_k U_{-2} B_{-2} h_k}{\phi_m U_2 B_2 h_m} + \inpr[p]{\phi_k U_{-2} B_{-2} h_k}{\phi_m U_0 B_0 h_m} \right\}
 \end{split}
 \end{align}
Now, consider the terms of \eqref{del22Mnph} individually.
\begin{enumerate}[label=(\Roman*)]
\item 
\begin{align}\label{estmt106}
 \begin{split}
 & \sum_{k,m=0}^\infty \inpr[p]{\phi_k U_2 B_2 h_k}{\phi_m U_2 B_2 h_m}\\
 & =\sum_{k,m,i=0}^\infty (2i+1)^{2p}  \inpr[]{\phi_k U_2 B_2 h_k}{h_i} \inpr[]{\phi_m U_2 B_2 h_m}{h_i} \\
 & = \sum_{i=0}^\infty (2i+1)^{2p} \phi_{i+2}^2 \frac{(i+2)(i+1)}{4} \\
 & = \sum_{l=0}^\infty (2l-3)^{2p} \phi_l^2 \frac{l(l-1)}{4} 
 \end{split}
 \end{align}
We note here that the terms for $l = 0$ and $l = 1$ do not contribute to the above sum. However, we carry these terms for notational convenience.
 \item 
\begin{align}\label{estmt107}
 \begin{split}
& \sum_{k,m=0}^\infty \inpr[p]{\phi_k U_0 B_0 h_k}{\phi_m U_0 B_0 h_m}\\
 & =\sum_{k,m,i=0}^\infty (2i+1)^{2p} \inpr[]{\phi_k U_0 B_0 h_k}{h_i} \inpr[]{\phi_m U_0 B_0 h_m}{h_i} \\
& = \sum_{i=0}^\infty (2i+1)^{2p} \phi_i^2 \frac{(2i+1)^2}{4}
 \end{split}
 \end{align}

 \item 
\begin{align}\label{estmt108}
 \begin{split}
& \sum_{k,m=0}^\infty \inpr[p]{\phi_k U_{-2} B_{-2} h_k}{\phi_m U_{-2} B_{-2} h_m} \\ 
& = \sum_{k,m,i=0}^\infty (2i+1)^{2p} \inpr[]{\phi_k U_{-2} B_{-2} h_k}{h_i} \inpr[]{\phi_m U_{-2} B_{-2} h_m}{h_i} \\
& = \sum_{i=0}^\infty (2i+1)^{2p} \phi_{i-2}^2 \frac{i(i-1)}{4} \\
& = \sum_{l=0}^\infty (2l+5)^{2p} \phi_l^2 \frac{(l+2)(l+1)}{4}
 \end{split}
 \end{align} 

\item 
 \begin{align}\label{estmt109}
 \begin{split}
 & \sum_{k,m=0}^\infty \inpr[p]{\phi_k U_2 B_2 h_k}{\phi_m U_0 B_0 h_m} \\
& = \sum_{k,m,i=0}^\infty (2i+1)^{2p} \inpr[]{\phi_k U_2 B_2 h_k}{h_i} \inpr[]{\phi_m U_0 B_0 h_m}{h_i} \\
& = - \sum_{i=0}^\infty (2i+1)^{2p} \phi_{i+2} \phi_i \frac{(2i+1)\sqrt{(i+2)(i+1)}}{4} 
 \end{split}
 \end{align}

 \item 
 \begin{align}\label{estmt110}
 \begin{split}
& \sum_{k,m=0}^\infty \inpr[p]{\phi_k U_2 B_2 h_k}{\phi_m U_{-2} B_{-2} h_m}\\ 
& = \sum_{k,m,i=0}^\infty (2i+1)^{2p} \inpr[]{\phi_k U_2 B_2 h_k}{h_i} \inpr[]{\phi_m U_{-2} B_{-2} h_m}{h_i} \\
& = \sum_{i=0}^\infty (2i+1)^{2p} \phi_{i+2} \frac{\sqrt{(i+2)(i+1)}}{2} \phi_{i-2} \frac{\sqrt{i(i-1)}}{2} \\
& = \sum_{l=0}^\infty (2l+5)^{2p} \phi_{l} \phi_{l+4} \frac{\sqrt{(l+1)(l+2)(l+3)(l+4)}}{4}
 \end{split}
 \end{align}

 \item 
 \begin{align}\label{estmt111}
 \begin{split}
 & \sum_{k,m=0}^\infty \inpr[p]{\phi_k U_0 B_0 h_k}{\phi_m U_2 B_2 h_m} \\
 & = \sum_{k,m,i=0}^\infty (2i+1)^{2p} \inpr[]{\phi_k U_0 B_0 h_k}{h_i} \inpr[]{\phi_m U_2 B_2 h_m}{h_i} \\
 & = - \sum_{i=0}^\infty (2i+1)^{2p} \phi_i \phi_{i+2} \frac{(2i+1)\sqrt{(i+2)(i+1)}}{4}
 \end{split}
 \end{align}

 \item 
 \begin{align}\label{estmt112}
 \begin{split}
 & \sum_{k,m=0}^\infty \inpr[p]{\phi_k U_0 B_0 h_k}{\phi_m U_{-2} B_{-2} h_m}\\
& = \sum_{k,m,i=0}^\infty (2i+1)^{2p} \inpr[]{\phi_k U_0 B_0 h_k}{h_i} \inpr[]{\phi_m U_{-2} B_{-2} h_m}{h_i} \\
& = - \sum_{i=0}^\infty (2i+1)^{2p} \phi_i \frac{2i+1}{2} \phi_{i-2} \frac{\sqrt{i(i-1)}}{2} \\
& = - \sum_{l=0}^\infty (2l+5)^{2p} \phi_{l+2} \phi_l \frac{(2l+5)\sqrt{(l+2)(l+1)}}{4}
 \end{split}
 \end{align}

  \item 
 \begin{align}\label{estmt113}
 \begin{split}
& \sum_{k,m=0}^\infty \inpr[p]{\phi_k U_{-2} B_{-2} h_k}{\phi_m U_2 B_2 h_m}\\
& = \sum_{k,m,i=0}^\infty (2i+1)^{2p} \inpr[]{\phi_k U_{-2} B_{-2} h_k}{h_i} \inpr[]{\phi_m U_2 B_2 h_m}{h_i} \\
& = \sum_{i=0}^\infty (2i+1)^{2p} \phi_{i-2} \frac{\sqrt{i(i-1)}}{2} \phi_{i+2} \frac{\sqrt{(i+2)(i+1)}}{2} \\
& = \sum_{l=0}^\infty (2l+5)^{2p} \phi_l \phi_{l+4} \frac{\sqrt{(l+1)(l+2)(l+3)(l+4)}}{4}
 \end{split}
 \end{align}

 \item 
 \begin{align}\label{estmt114}
 \begin{split}
& \sum_{k,m=0}^\infty \inpr[p]{\phi_k U_{-2} B_{-2} h_k}{\phi_m U_0 B_0 h_m}\\
& = \sum_{k,m,i=0}^\infty (2i+1)^{2p} \inpr[]{\phi_k U_{-2} B_{-2} h_k}{h_i} \inpr[]{\phi_m U_0 B_0 h_m}{h_i} \\
& = - \sum_{i=0}^\infty (2i+1)^{2p} \phi_{i-2} \frac{\sqrt{i(i-1)}}{2} \phi_i \frac{2i+1}{2} \\
& = - \sum_{l=0}^\infty (2l+5)^{2p} \phi_l \phi_{l+2} \frac{(2l+5)\sqrt{(l+2)(l+1)}}{4}
 \end{split}
 \end{align}
 
\end{enumerate}

\textbf{Step 3:} We now justify the bound on the expression `$- \inpr[p]{\phi}{\partial^4\phi} + \|\partial^2\phi\|_p^2$'. To do this, we first add up all the terms in \eqref{estmt101}-\eqref{estmt105} and \eqref{estmt106}-\eqref{estmt114} and then collect terms involving $\phi_l^2, \phi_l\phi_{l+2}$ and $\phi_l\phi_{l+4}$ separately. We have,

\begin{equation}\label{good-terms-expansion}
\begin{split}
& - \inpr[p]{\phi}{\partial^4\phi} + \|\partial^2\phi\|_p^2 \\
& = \sum_{l=0}^\infty (2l+1)^{2p} \phi_l^2 \left[\frac{l^2}{4} \left\{\left( \frac{2l - 3}{2l + 1} \right)^{2p} + \left( \frac{2l + 5}{2l + 1} \right)^{2p} - 2\right\} \right.\\
&\hspace{3cm} + \frac{l}{4} \left\{- \left( \frac{2l - 3}{2l + 1} \right)^{2p} + 3 \left( \frac{2l + 5}{2l + 1} \right)^{2p} - 2\right\} \\
&\hspace{3cm} + \left.\frac{1}{2} \left\{\left( \frac{2l + 5}{2l + 1} \right)^{2p} - 1\right\} \right] \\
&+ \sum_{l=0}^\infty (2l+1)^{2p} \phi_l \phi_{l+2} \sqrt{(l+1)(l+2)} \left[1 - \left( \frac{2l + 5}{2l + 1} \right)^{2p} \right]\\
&+ \sum_{l=0}^\infty  (2l+1)^{2p}\, \phi_l\phi_{l+4}\, \frac{\sqrt{(l+4)(l+3)(l+2)(l+1)}}{4} \left[2\left( \frac{2l + 5}{2l + 1} \right)^{2p} - 1 - \left( \frac{2l + 9}{2l + 1} \right)^{2p} \right]\\
&= \sum_{l=0}^\infty  (2l+1)^{2p} \inpr[]{\phi}{h_l} \left[ a_l \inpr[]{\phi}{h_l} + b_l \inpr[]{\phi}{U_2 h_l} + c_l \inpr[]{\phi}{U_4 h_l} \right],
\end{split}
\end{equation}
where the sequences $\{a_l\}_{l = 0}^\infty, \{b_l\}_{l = 0}^\infty$ and $\{c_l\}_{l = 0}^\infty$ are defined by
\begin{align*}
  a_l &:= \frac{l^2}{4} \left\{\left( \frac{2l - 3}{2l + 1} \right)^{2p} + \left( \frac{2l + 5}{2l + 1} \right)^{2p} - 2\right\}\\
&+ \frac{l}{4} \left\{- \left( \frac{2l - 3}{2l + 1} \right)^{2p} + 3 \left( \frac{2l + 5}{2l + 1} \right)^{2p} - 2\right\} \\
&+ \frac{1}{2} \left\{\left( \frac{2l + 5}{2l + 1} \right)^{2p} - 1\right\},\\
  b_l &:= \sqrt{(l+1)(l+2)} \left[1 - \left( \frac{2l + 5}{2l + 1} \right)^{2p} \right],\\
  c_l &:= \frac{\sqrt{(l+4)(l+3)(l+2)(l+1)}}{4} \left[2\left( \frac{2l + 5}{2l + 1} \right)^{2p} - 1 - \left( \frac{2l + 9}{2l + 1} \right)^{2p} \right].
\end{align*}
In the expression for $a_l$, the term $\left( \frac{2l - 3}{2l + 1} \right)^{2p}$ for $l = 0$ and $l = 1$ should be treated as $0$ (see \eqref{estmt106}). By Lemma \ref{fn-bound-lemma}, the above sequences are bounded. Also, by Lemma \ref{shift-operator}, $U_2$ and $U_4$ are bounded linear operators on $\Sc_p$. Hence, the result follows from \eqref{good-terms-expansion}.
\end{proof}

\begin{proof}[Proof of Corollary \ref{L-A-Monotonicity}]
Note that 
\[2\inpr[p]{\phi}{L\phi} + \|A\phi\|_{HS(p)}^2 = \kappa^2 \left( - \inpr[p]{\phi}{\partial^4\phi} + \|\partial^2\phi\|_p^2 \right) + \left( \inpr[p]{\phi}{\sigma^2\, \partial^2\, \phi - 2b\, \partial\, \phi} + \|-\sigma\partial\phi\|_{HS(p)}^2 \right).\]
The result follows as a consequence of Theorem \ref{4th-order-monotonicity} and \cite[Theorem 2.1]{MR2590157}.
\end{proof}

\section{Proofs of Theorem \ref{exist-unique-soln-spde} and Theorem \ref{exist-unique-soln-pde}}\label{S4:proof-existence-uniqueness}

\begin{proof}[Proof of Theorem \ref{exist-unique-soln-spde}]
Note that the operators $A_1, A_2$ and $L$ leave $\Sc$ invariant. Moreover, by Lemma \ref{L-A-bound-operator}, $A_1, A_2$ and $L$ are bounded linear operators from $\Sc_p$ to $\Sc_{p - 2}$.

By Corollary \ref{L-A-Monotonicity}, we have 
\[2\inpr[p]{\phi}{L\phi} + \|A\phi\|_{HS(p)}^2 \leq C \|\phi\|_p^2, \,\forall \phi \in \Sc\]
and
\begin{equation}\label{L-A-monotonicity-p-2}
2\inpr[p-2]{\phi}{L\phi} + \|A\phi\|_{HS(p-2)}^2 \leq C \|\phi\|_{p - 2}^2, \,\forall \phi \in \Sc.
\end{equation}
By \cite[Theorem 1]{MR2479730}, the Stochastic PDE \eqref{bilapspdeß1} has a unique $\Sc_p$ valued strong solution with equality in $\Sc_{p - 2}$. We note that the first monotonicity inequality is used to show the existence of a strong solution, and the second one for uniqueness.
\end{proof}

\begin{proof}[Proof of Theorem \ref{exist-unique-soln-pde}]
This proof is similar to \cite[Theorem 4.3 and Theorem 4.4]{MR2373102}.

Let $\{X_t\}_t$ denote the solution obtained in the proof of Theorem \ref{exist-unique-soln-spde}. Here, we have used \cite[Theorem 1]{MR2479730}, which in particular implies that 
\[\sup_{t \in [0, T]} \|X_t\|_{p}^2 < \infty, \, \forall T > 0.\]
Using the boundedness of the operator $A$ (see Lemma \ref{L-A-bound-operator}), we have the martingale property of the process $\{\int_0^t A(X_s)\cdot dB_s\}_t$. Setting $u_t := \Exp X_t, \forall t \geq 0$ and taking expectation in \eqref{integral-equality-spde}, we conclude that $\{u_t\}_t$ is a strong solution of the PDE \eqref{bilappde}.

To prove the uniqueness, let $\{\tilde u_t\}_t$ be another $\Sc_p$ valued strong solution with equality in $\Sc_{p - 2}$. Then, by the Monotonicity inequality \eqref{L-A-monotonicity-p-2}, we have
\begin{align*}
\|u_t - \tilde u_t\|_{p - 2}^2 &= 2\int_0^t \left[ \inpr[p - 2]{u_s - \tilde u_s}{L(u_s - \tilde u_s)}\right] \, ds\\
&\leq \int_0^t \left[ 2\inpr[p - 2]{u_s - \tilde u_s}{L(u_s - \tilde u_s)} + \|A(u_s - \tilde u_s)\|_{HS(p - 2)}^2\right] \, ds\\
&\leq C\int_0^t \|u_s - \tilde u_s\|_{p - 2}^2 \, ds
\end{align*}
for some constant $C > 0$.
The uniqueness then follows by the Gronwall's inequality.
\end{proof}

\textbf{Acknowledgement:} Suprio Bhar acknowledges the support of the Matrics grant MTR/2021/000517 from the Science and Engineering Research Board (Department of Science \& Technology, Government of India). Barun Sarkar acknowledges the support of SERB project -  SRG/2022/000991, Government of India.

\bibliographystyle{amsplain}
\bibliography{references}
\end{document}